\documentclass[11pt]{amsart}
\usepackage{amscd, amsmath, amssymb, amsthm, amsfonts, amstext, verbatim, enumitem, mathtools, xfrac, microtype, nameref, thmtools, hyperref, cleveref}
\usepackage[utf8]{inputenc}
\usepackage[backend=bibtex, citestyle=numeric-comp, natbib=true, sortlocale=en_US, url=true, doi=true, eprint=true, maxbibnames=4]{biblatex}            
\usepackage[english]{babel} 
\usepackage{stmaryrd}
\usepackage[matrix,arrow,curve]{xy}

\theoremstyle{plain} \declaretheorem[name=Theorem, Refname={Theorem,Theorems}]{tm} \Crefname{tm}{Theorem}{Theorems}
\numberwithin{equation}{section}

\newtheorem{lm}{Lemma} \numberwithin{lm}{section} \Crefname{lm}{Lemma}{Lemmas}
 \Crefname{cl}{Corollary}{Corollaries}
\newtheorem{prop}[lm]{Proposition} \Crefname{prop}{Proposition}{Propositions}
\newtheorem*{tm*}{Theorem}
\newtheorem*{lm*}{Lemma}
\theoremstyle{definition} \newtheorem{df}[lm]{Definition} \Crefname{df}{Definition}{Definitions}
 \Crefname{example}{Example}{Examples}
\theoremstyle{remark} \newtheorem{rk}[lm]{Remark} \Crefname{rk}{Remark}{Remarks}

\newcommand{\Ker}{\mathop{\mathrm{Ker}}\nolimits}

\newcommand{\E}{{\mathrm{E}}}
\newcommand{\GG}{{\mathrm{G}_{\mathrm{sc}}}}
\newcommand{\Um}{\mathop{\mathrm{Um}}\nolimits}
\newcommand{\St}{\mathop{\mathrm{St}}\nolimits}

\newcommand{\GL}{\mathop{\mathrm{GL}}\nolimits}

\newcommand{\Kt}{\mathop{\mathrm{K_2}}\nolimits}
\newcommand{\Ko}{\mathop{\mathrm{K_1}}\nolimits}

\newcommand{\epi}{\twoheadrightarrow}

\newcommand{\inv}{^{-1}}

\newcommand{\rA}{\mathsf{A}}
\newcommand{\rB}{\mathsf{B}}
\newcommand{\rC}{\mathsf{C}} 
\newcommand{\rD}{\mathsf{D}} 
\newcommand{\rE}{\mathsf{E}}

\title{On centrality of even orthogonal $\Kt$}

\author {Andrei Lavrenov}
\email {avlavrenov {\it at} gmail.com}

\author{Sergey Sinchuk}
\email{sinchukss {\it at} gmail.com}

\date {\today}
\thanks{The authors of the present paper acknowledge the financial support of the Russian Science Foundation grant 14-11-00297}

\begin{document}

\begin{abstract} We give a short uniform proof of centrality of $\Kt(\Phi, R)$ for all simply-laced root systems $\Phi$ of rank $\geq 3$.
\end{abstract}

\maketitle

\section*{Introduction}
The aim of this paper is to establish Quillen's local-global principle and the centrality of even orthogonal $\Kt$ for an arbitrary commutative ring $R$.

Recall that to each reduced irreducible root system $\Phi$ and a commutative ring $R$ one can associate a split simple simply-connected group $\GG(\Phi, R)$ called a \emph{Chevalley group}.
The {\it elementary group} $\E(\Phi, R)$ is defined as the subgroup of $\GG(\Phi, R)$ generated by elementary root unipotents $t_\alpha(\xi)$, $\xi\in R$, $\alpha\in \Phi$, see~\cite{Ta, St78}.
By definition, the Steinberg group $\St(\Phi, R)$ is presented by generators $x_\alpha(\xi)$, which model the unipotents $t_\alpha(\xi)$, and Steinberg relations (see~\cref{sec:prelim}). 
Thus, there is a well-defined map $\phi\colon\St(\Phi, R)\to \GG(\Phi, R)$ sending $x_\alpha(\xi)$ to $t_\alpha(\xi)$ whose image is $\E(\Phi, R)$.

By a theorem of Taddei (see~\cite{Ta}) the elementary group $\E(\Phi, R)$ is a normal subgroup of $\GG(\Phi, R)$ when $\Phi$ has rank at least $2$.
Thus, similarly to the definition of algebraic $\mathrm{SK}_1$ and $\Kt$-functors, one can define the groups $\Ko(\Phi, R)$, $\Kt(\Phi, R)$ as the cokernel and kernel of $\phi$:
$$\xymatrix{1 \ar[r] & \Kt(\Phi, R) \ar[r] & \St(\Phi, R) \ar[r] & \GG(\Phi, R) \ar[r] & \Ko(\Phi, R) \ar[r] & 1.}$$

It is a classical theorem of M.~Kervaire and R.~Steinberg that the stable Steinberg group $\St(R)$ is the universal central extension of $\E(R)$, see \cite[Theorem~III.5.5]{Kbook}.
The unstable analogue of the ``universality'' part of this theorem has been known since 1970's for $\Phi$ of rank $\ell \geq 5$ (or even $\ell \geq 4$ for $\Phi=\rA_\ell, \rB_\ell, \rC_\ell$), see \cite{St1, KaS}.
However, until recently, the analogue of the ``centrality'' part was only known in the special case $\Phi=\rA_\ell$, $\ell\geq 3$, see~\cite[Corollary~2]{vdK}.
The best known result for other classical $\Phi$ was that the centrality holds ``in the stable range'',
i.\,e. when the dimension of $R$ is suffiently small as compared to the rank of $\Phi$, see~\cite[Corollary~3.4]{St78}.

Recently the authors of this paper have established the centrality of $\Kt(\rC_\ell, R)$ ($\ell\geq 3$) and $\Kt(\rE_\ell, R)$ ($\ell=6,7,8$) for an arbitrary commutative ring $R$, see \cite{Lav, SCh}.
In turn, the key result of this paper is the following theorem.
\begin{tm}[``Centrality of $\Kt$''] \label{tm:centrality}  Let $R$ be an arbitrary commutative ring.
Then, for $\ell\geq 4$ the group $\Kt(\rD_\ell, R)$ is a central subgroup of $\St(\rD_\ell, R)$. \end{tm}

Centrality of $\Kt$ is a corollary of the following analogue of Quillen's local-global principle,
very much in the same fashion as the normality theorem for the elementary subgroup is a corollary of the local-global principle for $\Ko$, cf.~\cite[Theorem~3.1]{Basu05}.
\begin{tm}[``Quillen's LG-principle for $\Kt$''] \label{tm:lg-principle} 
Let $\Phi$ be an irreducible simply-laced root system of rank $\geq 3$ and $R$ be a commutative ring.
Then an element $g\in\St(\Phi,\,R[X])$ satisfying $g(0)=1\in\St(\Phi,\,R)$ is trivial in $\St(\Phi,\,R[X])$
if and only if the images $\lambda_M^*(g)\in\St(\Phi,\,R_M[X])$ of $g$ under localisation maps $\lambda_M\colon R\to R_M$ are trivial for all maximal ideals $M\trianglelefteq R$. \end{tm}

In the special case $\Phi=\rA_\ell$, $\ell\geq 4$, the assertion of \cref{tm:lg-principle} was obtained by Tulenbaev, see~\cite[Theorem~2.1]{Tul}.
Despite the fact that the key ingredient in Tulenbaev's proof (``another presentation'' of the Steinberg group) is available for $\Phi=\rA_3$, Tulenbaev's proof does not work in this case for technical reasons.
On the other hand, in~\cite{SCh} the second-named author developed a patching technique which allowed him to deduce \cref{tm:lg-principle} from Tulenbaev's result in another special case $\Phi=\rE_\ell$, $\ell=6,7,8$.
Thus, \cref{tm:lg-principle} contains two cases that have not been previously known: $\Phi=\rA_3$ and $\Phi=\rD_\ell$, $\ell\geq 4$.
In rank $2$ there are known counterxamples to centrality of $\Kt$, which shows that our assumption on the rank of $\Phi$ in \cref{tm:lg-principle} is the best possible, see~\cite{W}.
We also note that recently the first-named author has obtained an analogue of \cref{tm:lg-principle} for $\Kt(\rC_\ell, R)$ under the assumption $\ell\geq 3$, see~\cite{Lav2}.

The paper is organised as follows. 
After introducing the preliminaries we explain in \cref{sec:patching} how the patching technique of~\cite{SCh} allows one to reduce \cref{tm:lg-principle} to the special case $\Phi=\rA_3$.
Next, in \cref{sec:yap} we describe a new ``transpose-symmetric'' presentation for the relative linear Steinberg group.
Finally, this presentation is used in~\cref{sec:lgp} to prove \cref{tm:lg-principle} in the case $\Phi=\rA_3$.

\section{Preliminaries} \label{sec:prelim}
Throughout this paper $R$ denotes an associative commutative ring with identity.
All commutators are left-normed, i.\,e. $[x,\,y]=xyx\inv y\inv$. 
We denote by $R^n$ the free right $R$-module with a basis $e_1,\ldots,e_n$ and by $\Um(n,\,R)$ the set of unimodular columns $v\in R^n$, i.\,e. 
columns whose components generate $R$ as an ideal.

First of all, recall that for a simply laced root system $\Phi$ of rank $\geq 2$ the Steinberg group $\St(\Phi, R)$ 
is defined by generators $x_\alpha(r)$, $\alpha\in\Phi$, $r\in R$ and the following relations:
\begin{align}
 x_\alpha(r) x_\alpha(s) & = x_\alpha(r+s), & \nonumber \\
 [x_\alpha(r),  x_\beta(s)] & = x_{\alpha+\beta}(N_{\alpha,\beta}\cdot rs), &\ \alpha,\beta\in \Phi,\ \alpha+\beta\in\Phi, \nonumber \\
 [x_\alpha(r),  x_\beta(s)] & = 1, &\ \alpha,\beta\in \Phi,\ \alpha+\beta \not \in \Phi \cup \{0\}. \nonumber
\end{align}
Here $N_{\alpha,\beta}$ are the structure constants of $\Phi$ which are equal to $\pm 1$.
We use a more familiar notation for root unipotents in the case $\Phi=\rA_\ell$:
$x_{ij}(r),\ 1\leq i\neq j\leq n,\ r\in R,$
\setcounter{equation}{0}
\renewcommand{\theequation}{S\arabic{equation}}
\begin{align}
x_{ij}(r)x_{ij}(s)      & = x_{ij}(r+s),& \label{add0}\\
[x_{ij}(r),\,x_{hk}(s)] & = 1,& \text{ for }h\neq j,\ k\neq i, \label{ccf1}\\
[x_{ij}(r),\,x_{jk}(s)] & = x_{ik}(rs).& \label{ccf2}
\end{align}
In this situation the elementary root unipotents coincide with elementary transvections $t_{ij}(r)=1+r \cdot e_{ij}$, $1\leq i\neq j\leq n$, $r\in R$
and the elementary subgroup $\E(\rA_{n-1}, R)$ is denoted by $\E(n, R)$.
Here $1$ denotes the identity matrix and $e_{ij}$ is the standard matrix unit. 
The natural projection $\phi\colon\St(n,\,R)\rightarrow\E(n,\,R)$ sends $x_{ij}(r)$ to $t_{ij}(r)$.

In \cite{Sus} Suslin showed that for $n\geq 3$ the elementary group $\E(n, R)$ coincides with the subgroup of $\GL(n,\,R)$ generated by matrices
of the form $t(u,\,v)=1+uv^t$ where $u\in\Um(n,\,R)$, $v\in R^n$ and $u^tv=0$. Here $u^t$ stands for the transpose of $u$.
This result clearly implies that $\E(n,\,R)$ is normal in $\GL(n,\,R)$.
Developing Suslin's ideas, van~der~Kallen showed that for $n\geq4$ the Steinberg group $\St(n,\,R)$ is isomorphic to the group with the following presentation by generators
$$\{X(u,\,v)\mid u\in\Um(n,\,R),\ v\in R^n,\ u^tv=0\}$$ and relations:
\setcounter{equation}{0} \renewcommand{\theequation}{K\arabic{equation}}
\begin{align}
&X(u,\,v)X(u,\,w)=X(u,\,v+w), \label{add1} \\
&X(u,\,v)X(u',\,v')X(u,\,v)\inv=X(t(u,\,v)u',\,t(u,\,v)^* v'). \label{conj1}
\end{align}
In the above formula $t(u,\,v)^*$ denotes the corresponding contragradient matrix, i.\,e. $t(u,\,v)^* = t(v,\,u)\inv = t(v,\,-u)$.
We employ the same notation for contragradient matrices below, i.\,e. for $M\in\E(n,R)$ we set $M^* = (M^t)\inv$.

The above presentation clearly implies that $\phi\colon\St(n,\,R)\rightarrow\E(n,\,R)$ is a central extension.
Notice that we parametrize elements $X(u, v)$ by two columns rather than by a column and a row as it is done in~\cite{vdK}.
Our generator $X(u,v)$ corresponds to van der Kallen's generator $X(u, v^t)$.

Recall that an ideal $I\trianglelefteq R$ is called a \emph{splitting ideal} if the canonical projection $R \twoheadrightarrow R/I$ splits as a unital ring morphism, i.\,e. one has $\pi\sigma=1$:
$$\xymatrix{0\ar@<0.1ex>[r]^{}&I\ar@<0.1ex>[r]^{}&R\ar@<0.5ex>[r]^{\pi}&R/I\ar@<0.5ex>[l]^{\sigma}\ar@<0.1ex>[r]^{}&0}$$
For a splitting ideal $I\trianglelefteq R$ the \emph{relative Steinberg group} $\St(\Phi, R, I)$ can be defined as the kernel $\Ker\big(\St(\Phi,\,R)\epi\St(\Phi,\,R/I)\big)$.
Applying the functor $\St(\Phi, -)$ to $\pi$ and $\sigma$ we get that $\St(\Phi, R) = \St(\Phi, R, I) \rtimes \St(\Phi, R/I)$.
Here the group $\St(\Phi, R/I)$ acts on $\St(\Phi, R, I)$ by conjugation via the splitting map, i.\,e. $(g, h) (g',h') = (g {}^{\sigma^*(h)}g', hh')$.

In the case of an arbitrary $I$ the relative Steinberg group $\St(\Phi, R, I)$ is not a subgroup of $\St(\Phi, R)$ and should be defined as a central extension of the above kernel.
For the purposes of the present text it suffices to consider only relative Steinberg groups corresponding to splitting ideals.
For more information regarding the general case we refer the reader to \cite[Section~3]{SCh}.

\section{Reduction to the case \texorpdfstring{$\rA_3$}{A3}} \label{sec:patching}
First, we make the following definition inspired by Lemmas~2.3,~3.2 of~\cite{Tul}.

\begin{df} \label{df:tlp}
We say that the Steinberg group functor $\St(\Phi, -)$ satisfies {\it Tulenbaev lifting property} if for any ring $R$ and any non-nilpotent element $a \in R$
there exists a map $T_\Phi$ which fits into the following commutative diagram.
\numberwithin{equation}{section}
\setcounter{equation}{1}
\begin{equation} \xymatrix{\St(\Phi,\,R\ltimes XR_a[X],\,XR_a[X]) \ar@{^{(}->}[r] \ar[d] & \St(\Phi, R \ltimes XR_a[X]) \ar[d]_-{\lambda_a^*} \\
                           \St(\Phi,\,R_a[X],\,XR_a[X]) \ar@{-->}[ru]_-{T_\Phi} \ar@{^{(}->}[r] & \St(\Phi, R_a[X])} \label{eq:diagram} \end{equation}
Here, $R\ltimes XR_a[X]$ denotes the semidirect product of the ring $R$ and the $R$-algebra $XR_a[X]$ 
(see~e.\,g.~\cite[Definition~3.2]{SCh}), $\lambda_a^*$ is the morphism of Steinberg groups induced by the morphism $\lambda_a\colon R \to R_a$ of localisation by the powers of $a$ and 
the left arrow is the restriction of $\lambda_a^*$ to the first factor in the decomposition $\St(\Phi, R \ltimes XR_a[X]) \cong \St(\Phi, R\ltimes XR_a[X], XR_a[X]) \rtimes \St(\Phi, R)$.
Notice that $XR_a[X]$ is a splitting ideal for both $R_a[X]$ and $R \ltimes XR_a[X]$.
\end{df}
In fact, one can show that Tulenbaev lifting property implies that the relative Steinberg groups in the left-hand side of the diagram are isomorphic.
To see this, notice that the commutativity of the top (resp. bottom) triangle implies that left arrow is injective (resp. surjective).

We claim that the main results of this paper follow from Tulenbaev's lifting property.
\begin{prop} For any irreducible $\Phi$ of rank $\geq 2$ hold implications \[a) \implies b) \implies c) \implies d).\]
\begin{enumerate}
 \item The Steinberg group functor $\St(\Phi, -)$ satisfies lifting property~\eqref{eq:diagram};
 \item A dilation principle holds for $\St(\Phi, -)$, i.\,e. if $g\in\St(\Phi, A[X], XA[X])$ is such that $\lambda_a^*(g) = 1 \in \St(\Phi, R_a[X])$ then
       for a sufficiently large $n$ one has $ev_{X = a^n Y}^*(h) = 1 \in \St(\Phi, A[Y])$, where $ev_{X=a^n Y}$ denotes the unique $R$-algebra morphism $R[X]\to R[Y]$ sending $X$ to $a^nY$;
 \item The assertion of \cref{tm:lg-principle} holds for $\St(\Phi, -)$;
 \item $\Kt(\Phi, R)$ is a central subgroup of $\St(\Phi, R)$.
\end{enumerate} \end{prop}
\begin{proof} The specified implications are contained in the proofs of Lemma~15, Theorem~2 and Theorem~1 of \cite{SCh}, respectively. \end{proof}
     
In order to construct the arrow $T_\Phi$ fitting in the diagram~\eqref{eq:diagram} one needs to have some direct description of the relative Steinberg group $\St(\Phi, R, I)$.
For example, in~\cite{Tul} Tulenbaev uses a variant of W.~van~der~Kallen's ``another presentation'' for this purpose. The details of this approach are outlined in the next section.
On the other hand, in~\cite{SCh} the second-named author obtained a theorem which asserts that the relative Steinberg group $\St(\Phi, R, I)$ can be presented as an amalgamated product of
several copies of relative Steinberg groups of smaller rank. 

To formulate this theorem we will need one more definition.
For $s\in I$, $r\in R$ and $\alpha\in \Phi$ we set $z_\alpha(s, r) := x_\alpha(s)^{x_{-\alpha}(r)} \in \St(\Phi, R, I)$.
If the rank of $\Phi$ is at least $2$ then $\St(\Phi, R, I)$ is generated by the set consisting of all elements $z_\alpha(s,r)$, see~\cite[Lemma~5]{SCh}.

Now, denote by $G_0$ the free product of Steinberg groups $\St(\Psi, R, I)$ where $\Psi$ varies over all root subsystems of $\Phi$ which have type $\rA_3$.
Denote by $z_\alpha^\Psi(s,r)$ the image of $z_\alpha(s,r)\in\St(\Psi, R, I)$ under the canonical embedding $\St(\Psi, R, I) \hookrightarrow G_0$.
Now let $G$ be the quotient of $G_0$ modulo all relations of the form $z_\alpha^{\Psi_1}(s, r) = z_\alpha^{\Psi_2}(s, r)$, where $\alpha\in\Psi_1\cap\Psi_2$, $s\in I$, $r\in I$ 
and both $\Psi_1$ and $\Psi_2$ have type $\rA_3$.
The following result has been proved in~\cite{SCh}, see~Theorem~9 combined with Remark~3.16.
\begin{tm*} \label{tm:relPres} For a simply laced irreducible root system $\Phi$ of rank $\geq 3$ the canonical map $G \rightarrow \St(\Phi, R, I)$ is an isomorphism of groups.
\end{tm*}

The above theorem allows one to construct the arrow $T_\Phi$ by means of the universal property of colimits.
More precisely, one needs to construct the maps $T_{\Psi}$ for all subsystems $\Psi\subseteq\Phi$ of type $\rA_3$ and then check that
relations $T_{\Psi_1}(z_\alpha^{\Psi_1}(s,r)) = T_{\Psi_2}(z_\alpha^{\Psi_2}(s,r))$ hold for all $\Psi_1$, $\Psi_2$ of type $\rA_3$ containing a common root $\alpha$.
If in the latter statement one replaces $\rA_3$'s with $\rA_4$'s then such a check has already been made in~\cite{SCh}, see Lemma~14.
It is easy to see that the argument of Lemma~14 works verbatim in our situation (one only has to replace all references to Lemma~3 in its proof with references to Lemma~2).
On the other hand, the probem of existence of $T_\Psi$ for $\Psi$ of type $\rA_3$ is not easy and will be solved in the following sections of the paper.

Summing up the above discussion, we have shown that both \cref{tm:centrality} and \cref{tm:lg-principle} follow from the special case $\Phi=\rA_3$ of Tulenbaev's lifting property.

\section{Yet another presentation for the relative Steinberg group.} \label{sec:yap}
In order to construct the arrow $T$ Tulenbaev uses a variant of ``another presentation'' for the relative Steinberg group, cf.~\cite[Proposition~1.6]{Tul}. 
\setcounter{lm}{1}
\begin{prop}\label{prop:TulPres}
For a splitting ideal $I$ and $n\geq 4$ the group $\St(n,\,R,\,I)$ is isomorphic to the group defined by generators
$$\{X(u,\,v)\mid u\in\E(n,\,R)e_1,\ v\in I^n,\ u^tv=0\}$$ and the following set of relations:
\setcounter{equation}{0}
\renewcommand{\theequation}{T\arabic{equation}}
\begin{align}
&X(u,\,v)X(u,\,w)=X(u,\,v+w), \label{add2}\\
&X(u,\,v)X(u',\,v')X(u,\,v)\inv=X(t(u,\,v)u',\,t(v,\,u)\inv v'), \label{conj2}  \\
&X(ur+w,\,v)=X(u,\,vr)X(w,\,v)\,\text{ for }r\in R,\ (u,\,w)\in\Um_{n\times2}(R) \label{add3}.
\end{align}
Moreover, one can replace the last family of relations in the definition of $\St(n, R, I)$ with the following smaller family:
\setcounter{equation}{2} \renewcommand{\theequation}{T\arabic{equation}'}
\begin{equation} X(Me_1r+Me_2,\,M^*e_3a)=X(Me_1,\,M^*e_3ar)X(Me_2,\,M^*e_3a)\, \label{add3'} \end{equation} 
for $r\in R$, $a\in I$ and $M\in E(n, R)$.
\end{prop}
In the statement of the third relation we denote by $\Um_{n\times2}(R)$ the set of $n\times2$ \emph{unimodular matrices},
i.\,e. matrices $m$ such that there exists a $2\times n$ matrix $m'$ satisfying $m'm=\left(\begin{smallmatrix}1&0\\0&1\end{smallmatrix}\right)$.

\begin{proof} 
Since the statement of our proposition contains one more claim which is not contained in~\cite[Proposition~1.6]{Tul}, we reproduce below a more detailed version of the original proof of Tulenbaev.
For a time being denote by $G$ the group defined by T1, T2, T3'. We write $\pi(\xi)$ for the image of $\xi\in R$ in the quotient ring and assume that $R/I$ is embedded into $R$. The idea is to construct a pair of mutually inverse maps:
$$\xymatrix{G\rtimes \St(n, R/I) \ar@<0.6ex>[r]^-\varphi &  \ar[l]^-\psi \St(n, R)}$$
where $\St(n, R/I)$ acts on $G$ via $$\,^{X(u,\,v)}X(u',v')=X(t(u,\,v)u',\,t(u,\,v)^*v').$$
We define the map $\psi$ by $\psi(x_{ij}(\xi)) = (X(e_i, (\xi - e_j\pi(\xi)),\ x_{ij}(\pi(\xi)))$. 
We have to check that $\psi$ preserves relations \eqref{add0}--\eqref{ccf2}.
For example, let us check \eqref{ccf2}. 
First of all, recall that in any semidirect product of groups one can compute the commutator of two elements in the following way:
$$[(a, b), (c, d)] = (a \cdot {}^bc \cdot {}^{bdb^{-1}}a^{-1} \cdot {}^{[b, d]}c^{-1} , [b, d]).$$
For $\xi\in R$ set $\xi' = \xi - \pi(\xi)$.
Specializing $a = X(e_i,  e_j\xi')$, $b = x_{ij}(\pi(\xi))$, $c = X(e_j, e_k\eta')$, $d = x_{jk}(\pi(\eta))$ and using the definition of the conjugation action 
one computes the commutator $[\psi(x_{ij}(\xi), \psi(x_{jk}(\eta)]$ as follows:
\setcounter{equation}{2}
\begin{equation} \left(X(e_i,  e_j\xi') \cdot X(e_j + e_i\pi(\xi), e_k\eta') \cdot X(e_i,  e_k\pi(\eta)\xi' -  e_j\xi') \cdot X(e_j, -e_k\eta'),\ x_{ik}(\pi(\xi\eta))\right). \label{eq:proofS3}\end{equation}
\setcounter{lm}{3}
After applying additivity relations \eqref{add2}, \eqref{add3'} and moving the factor $X(e_i, - e_j\xi')$ to the left hand side of the formula by means of relation \eqref{conj2}
one gets the following expression:
$$\left(X(e_j, e_k\eta') \cdot X(e_i, e_k(\xi' \eta' + \pi(\xi)\eta' + \pi(\eta)\xi')) \cdot X(e_j, -e_k\eta'),\ x_{ik}(\pi(\xi\eta))\right).$$
Since $\xi'\eta' + \pi(\xi)\eta' + \pi(\eta)\xi' = \xi\eta - \pi(\xi\eta),$ the above expression simplifies to $\psi(x_{ik}(\xi\eta))$, as claimed.
Verification of the fact that $\psi$ preserves relations \eqref{add0}, \eqref{ccf1} is similar to the above computation but is easier.
In particular, one does not need to use the relation~\eqref{add3'}.

Now define the map $\varphi$ by $\varphi\left((X(u, v), 1)\right) = X(u, v),$ $u\in E(n, R)e_1$, $v\in I^n$, $\varphi\left((1, x_{ij}(\xi))\right) = X(e_i, e_j\sigma(\xi))$.
It is clear that $\varphi$ preserves the relations \eqref{add2}--\eqref{conj2}. The relation~\eqref{add3} follows from~\eqref{add1}--\ref{conj1}, see~\cite[(1.3)]{Tul}, thus, the map $\varphi$ is well-defined.
it is easy to check that $\varphi\psi=1$ therefore $\psi$ is injective. To see that $\psi$ is surjective it remains to notice that the elements $(X(u,v), 1)$, $u\in \E(n, R)e_1$, $v\in I^n$ lie in the image of $\psi$.
\end{proof}

Both ``another presentation'' of van der Kallen and that of Tulenbaev are given in terms of generators parametrised by pairs of vectors,
 where the first one is ``nice'' in some sense i.\,e. unimodular, a column of an elementary matrix or the like, while the second one is arbitrary.
It is easy to formulate the additivity property in the second argument for these generators, cf.~\eqref{add1}, \eqref{add2},
while it is not so easy when it comes to the additivity in the first argument, cf. \eqref{add3}.

The construction of the homomorphism $T\colon\St(n,\,R_a[X],\,XR_a[X])\rightarrow\St(n,\,R\ltimes XR_a[X])$
amounts to choosing certain elements in the group $\St(n,\,R\ltimes XR_a[X])$ and proving that these elements satisfy relations \eqref{add2}--\eqref{add3}, see~\cite[Lemmas~1.2 and~1.3\,c)]{Tul}.
The main problem with this recipe is that the assumption $n\geq5$ is essential for the verification of the fact that relations~\eqref{add3} (or even~\eqref{add3'}) hold.

To generalise Tulenbaev's results to $n=4$ we use a more symmetric presentation with two types of generators: 
$F(u,\,v)$ with $u$ nice and $v$ arbitrary and $S(u,\,v)$ with $u$ arbitrary and $v$ nice.
The generators $F(u,\,v)$ are additive in the second component, while $S(u,\,v)$ are additive in the first one.
When $u$ and $v$ are both nice we require that these two generators coincide.
More formally, we make the following definition.

\begin{df}
For $I\trianglelefteq R$ and $n\geq4$ define $\St^*(n,\,R,\,I)$ to be the group with the set of generators
\begin{multline*}
\{F(u,\,v)\mid u\in\E(n,\,R)e_1,\ v\in I^n,\ u^tv=0\}\,\cup\{S(u,\,v)\mid u\in I^n,\ v\in\E(n,\,R)e_1,\ u^tv=0\}
\end{multline*}
subject to the relations
\setcounter{equation}{0}
\renewcommand{\theequation}{R\arabic{equation}}
\begin{align}
&F(u,\,v)F(u,\,w)=F(u,\,v+w), \label{add4}\\
&S(u,\,v)S(w,\,v)=S(u+w,\,v), \label{add5}\\
&F(u,\,v)F(u',\,v')F(u,\,v)\inv=F(t(u,\,v)u',\,t(v,\,u)\inv v'), \label{conj3} \\
&F(u,\,va)=S(ua,\,v),\ \text{for all}\ a\in I,\,(u,\,v^t)=(M e_1,\,e_2^t M\inv),\,M\!\in\E(n,\,R) \label{coinc}.
\end{align}
\end{df}
Notice that we only described how a generator $F(u, v)$ acts via conjugation on another generator $F(u',v')$ and have omitted three similar relations involving generators $S(u,v)$ of the second type.
The reason for this is the following lemma asserting that the ``missing'' relations automatically follow from \eqref{add4}--\eqref{coinc}.
\begin{lm}
\label{allyouneedisf}
Denote by $\phi\colon\St^*(n,\,R,\,I)\rightarrow\E(n,\,R)$ the natural map sending $F(u,\,v)\mapsto t(u,\,v)$ and $S(u,\,v)\mapsto t(u,\,v)$.
Then the following holds:
\begin{enumerate}
\item $\St^*(n,\,R,\,I)$ is generated as an abstract group by the set of elements 
      $$\{F(u,\,va)\mid a\in I,\ (u,\,v^t)=(Me_1,\,e_2^tM\inv),\,M\in\E(n,\,R)\};$$
\item for any $g\in\St^*(n,\,R,\,I)$ one has 
      \setcounter{equation}{2} \renewcommand{\theequation}{R\arabic{equation}'}
      \begin{equation} gF(u,\,v)g\inv=F(\phi(g)u,\,\phi(g\inv)^tv); \end{equation}
\item for any $g\in\St^*(n,\,R,\,I)$ one has
       \setcounter{equation}{2} \renewcommand{\theequation}{R\arabic{equation}''}
      \begin{equation} gS(u,\,v)g\inv=S(\phi(g)u,\,\phi(g\inv)^tv); \end{equation}
\item there is a ``transpose automorphism'' defined on $\St^*(n, R, I)$ satisfying
      $$F(u,\,v)^t=S(v,\,u),\quad S(u,\,v)^t=F(v,\,u).$$
\end{enumerate} \end{lm}

\begin{proof}
Let $F(u,\,v)$ be an arbitrary generator of the first type and let $M\in\E(n,\,R)$ be such that $F(u,\,v)=F(Me_1,\,M^*\tilde v)$. 
By \eqref{add4} we have $$F(u,\,v)=\prod\limits_{k\neq1}F(Me_1,\,M^*e_k\tilde v_k)$$
where $\tilde v_k$ stands for the $k$-th coordinate of $\tilde v=\sum e_i\tilde v_i$.
Applying relations \eqref{add5}, \eqref{coinc} we get
$$S(u',\,v')=\prod\limits_{k\neq1}S(Ne_k\tilde u_k,\,N^*e_1)=\prod\limits_{k\neq1}F(Ne_k,\,N^*e_1\tilde u_k)$$
which implies $(\mathrm{a})$. Obviously, $(\mathrm{b})$ follows from $(\mathrm{a})$. 
To prove $(\mathrm{c})$ it suffices to show that
$$F(u,\,v)S(u',\,v')F(u,\,v)\inv=S(t(u,\,v)u',\,t(u,\,v)^*v').$$
Indeed, for $S(u',\,v')=\prod_{k\neq 1} F(Ne_k,\,N^*e_1\tilde u_k)$ we have
\begin{multline*}
F(u,\,v)S(u',\,v')F(u,\,v)\inv=F(u,\,v)\prod F(Ne_k,\,N^*e_1\tilde u_k)F(u,\,v)\inv=\\
=\prod F(t(u,\,v)Ne_k,\,t(u,\,v)^*N^*e_1\tilde u_k)=\\
=\prod S(t(u,\,v)Ne_k\tilde u_k,\,t(u,\,v)^*N^*e_1)=S(t(u,\,v)u',\,t(u,\,v)^*v').
\end{multline*}
Finally, $(\mathrm{d})$ follows from $(\mathrm{c})$.
\end{proof}

Our next goal is to show that for a splitting ideal $I\trianglelefteq R$ the group $\St^*(n,\,R,\,I)$ is isomorphic to $\St(n,\,R,\,I)=\Ker\big(\St(n,\,R)\epi\St(n,\,R/I)\big)$.
With this end, we construct two mutually inverse homomorphisms
$$\xymatrix{\St^*(n,\,R,\,I)\ar@<0.5ex>[r]^{\iota}&\St(n,\,R,\,I)\ar@<0.5ex>[l]^{\kappa}.}$$
In order to construct the arrow $\kappa$ we use the presentation from Proposition~\ref{prop:TulPres}.
We set $$\kappa(X(u,v)) = F(u,v),\ u\in \E(n, R)e_1,\ v\in I^n.$$ 
It suffices to check that $F(u,v)$ satisfy the relations \eqref{add2}, \eqref{conj2}, \eqref{add3'}.
The validity of relations \eqref{add2}, \eqref{conj2} is obvious, while \eqref{add3'} follows from \eqref{add5} and \eqref{coinc}. Indeed,
\begin{multline} F(Me_1, M^*e_3ar) F(Me_2, M^*e_3a) = S(Me_1ar, M^*e_3) S(Me_2a, M^*e_3) = \\ = S(M(e_1r+e_2)a, M^*e_3) = F(Me_1r+Me_2, M^*e_3a). \qedhere \end{multline}
As for the map $\iota$, we first define it as a homomorphism to the absolute group. 
Then we prove that the image of $\iota$ is contained in $\St(n, R, I)$.
Clearly, the elements $F(u,\,v)$ should go to van der Kallen's generators $X(u,\,v)$.
But now we should find the images for the elements $S(u,\,v)$ as well.
First of all, recall that van der Kallen in~\cite[3.8--3.10]{vdK} constructs the following elements:
$$x(u,\,v)\in\St(n,\,R),\ u^tv=0,\ u_i=0\,\text{ or }\,v_i=0\,\text{ for some }\,1\leq i\leq n.$$
The following definition is the ``transpose'' of \cite[3.13]{vdK}.
\begin{df} For $u\in R^n$, $v\in\E(n,\,R)e_1$, $u^tv=0$, consider the set $\overline Y(u,\,v)\subseteq\St(n,\,R)$
 consisiting of all elements $y\in\St(n,\,R)$ that admit a decomposition $y=\prod x(u^k,\,v)$, 
 where $\sum u^k=u$ and $u^k= (e_pv_q-e_qv_p)c_k$ for some $c_k\in R$, $1\leq p\neq q\leq n$. \end{df}

Since columns of elementary matrices are unimodular, one can find $w\in R^t$ such that $w^tv=1$.
It is not hard to check that \setcounter{equation}{6} \setcounter{lm}{7}
\begin{equation} (w^t\cdot v)\cdot u = \sum_{p<q}u_{pq},\text{ where }u_{pq} = (e_pv_q - e_qv_p)\cdot  (u_pw_q - u_qw_p)\in{}\!R^n, \label{eq:canonical}\end{equation}
therefore $\overline Y(u,\,v)$ is not empty, cf.~\cite[3.1--3.2]{vdK}.
Obviously, for $x\in\overline Y(u,\,v)$ and $y\in\overline Y(w,\,v)$ one has $xy\in\overline Y(u+w,\,v)$.

Repeating~\cite[3.14--3.15]{vdK} verbatim one can show the following.
\begin{lm} For $g\in\St(n,\,R)$ one has
\begin{enumerate} \item $g\overline Y(u,\,v)g\inv\subseteq\overline Y(\phi(g)u,\,\phi(g)^*v)$;
                  \item $\overline Y(u,\,v)$ consists of exactly one element. \end{enumerate} \end{lm}

We denote the only element of $\overline Y(u,\,v)$ by $Y(u,\,v)$. 
Now, we can finish the definition of $\iota$ by requiring that $\iota(S(u,\,v)) = Y(u,\,v)$ for $S(u,\,v)\in\St^*(n,\,R,\,I)$. 
To show that the map $\iota$ is well-defined we should check that elements $X(u,\,v)$ and $Y(u,\,v)$ satisfy relations \eqref{add4}--\eqref{coinc},
with $F$'s and $S$'s replaced with $X$'s and $Y$'s.
Here, only the relation \eqref{coinc} is not immediately obvious.
\begin{lm} \label{lm:XY} For $(u,\,v^t)=(Me_1,\,e_2^tM\inv)$, $M\in\E(n,\,R)$, $a\in I$, one has $X(u,\,va)=Y(ua,\,v)$. \end{lm}
\begin{proof}
In view of the above lemma we only need to show that $X(e_1,\,e_2a)=Y(e_1a,\,e_2)$.
This equality can be obtained by computing the commutator $[Y(-e_3,\,e_2),\,X(e_1,\,e_3a)]$ in two ways:
$$ Y(-e_3,\,e_2)X(e_1,\,e_3a)Y(-e_3,\,e_2)\inv\cdot X(e_1,\,-e_3a)=X(e_1,\,e_2a), \text{ and} $$
$$ Y(-e_3,\,e_2)\cdot\,X(e_1,\,e_3a)Y(e_3,\,e_2)X(e_1,\,e_3a)\inv=Y(e_1a,\,e_2). \qedhere $$
\end{proof}

Since $\pi^*(\iota(F(u, v))) = X(\pi(u),\,0)=1$ and $\pi^*(\iota(S(u,v)) = Y(0,\,\pi(v)) = 1$ we get that $\mathrm{Im}(\iota)\subseteq\Ker(\pi^*)=\St(n,\,R,\,I)$. 
It is clear that $\iota\kappa=\mathrm{id}$ hence $\kappa$ is injective.
On the other hand, $\kappa$ is surjective by Lemma~\ref{allyouneedisf}\,$(\mathrm a)$.
Thus, we have demonstrated the following result.

\begin{prop} \label{lm:map-iota}
 For a splitting ideal $I \trianglelefteq R$ and $n\geq 4$ the groups $\St^*(n,\,R,\,I)$ and $\St(n,\,R,\,I)$ are isomorphic.
\end{prop}
 
\section{The local-global principle for \texorpdfstring{$\Kt$}{K2}}\label{sec:lgp}
Now we turn to the main result of this paper, namely, for $n\geq 4$ we construct the map $$T\colon\St(n,\,R_a[X],\,XR_a[X])\rightarrow\St(n,\,R\ltimes XR_a[X]),$$ that fits into the diagram \eqref{eq:diagram}.

We will prove a somewhat more general result. Let $B$ be a commutative unital ring.
We call an ideal $I$ of $B$ \emph{uniquely $r$-divisible} if for every $m\in I$ there exists a unique $m'\in I$ such that $rm' = m$.
In the sequel we denote such an $m'$ by $\frac{m}{r}$.
Clearly, $I$ is uniquely $r$-divisible if and only if the restriction of the principal localisation morphism $\lambda_r\colon R \to R_r$ to $I$ is an isomorphism.

\begin{tm} \label{a3map} Let $B$ be a ring, $a\in B$ and let $I$ be an ideal of $B$ that is uniquely $a$-divisible.
Then for $n\geq 4$ there exists a map $T\colon\St^*(n,\,B_a,\,I)\rightarrow\St(n,\,B)$ that makes the following diagram commute:
\begin{equation} \xymatrix{ \St^*(n,\,B,\,I)\ar@<-0.0ex>[rr]^{\iota}\ar@<-0.0ex>[d]_{\lambda_{a}^*} && \St(n,\,B)\ar@<-0.0ex>[d]^{\lambda_{a}^*}\\
                            \St^*(n,\,B_a,\,I)\ar@<-0.0ex>[rr]_{\iota}\ar@{-->}[rru]_{T}            && \St(n,\,B_a)} \label{eq:a3diag} \end{equation} \end{tm}

For a vector $u\in R^n$ we denote by $I(u)$ the ideal generated by the components of $u\in R^n$, i.\,e. $I(u)=\sum\limits_{k=1}^nu_kR$.
To prove \cref{a3map} we follow the approach of Tulenbaev and construct certain family of elements $X_{u,v}(a)\in \St(n,\,B)$, $a\in B$ which contains van der Kallen's generators $X(u, v)$ as a subfamily, i.\,e. $X_{u,v}(1) = X(u, v)$.
The key feature of this definition is that the assumption $u \in\Um(n, B)$ is replaced by a weaker condition $a \in I(u)$ equivalent to imposing that $u$ becomes unimodular after localisation by the powers of $a$.

\setcounter{df}{1}
\begin{df} \label{df:TulX}
For $u \in R^n$ we denote by $D(u)$ the set consisting of all $v\in B^n$ decomposing into a sum $v=\sum_{k=1}^Nv_k$, where $v_1,\ldots,v_N\in B^n$ are such that $u^tv_k=0$ and each $v_k$ has at least two zero coordinates.
Now, for $a\in I(u)$ and $v \in D(u)$ set $X_{u,v}(a) = \prod\limits_{k=1}^Nx(u,\,v_ka)$.
\end{df}
Notice that Tulenbaev uses different notation for van der Kallen elements, e.\,g. he writes $X_{u,v}$ instead of $X(u,\,v)$ and $X(u,\,v)$ instead of $x(u,\,v)$. 
We stick to van der Kallen's notation.

Tulenbaev shows that the factors $x(u,\,v_ka)$, $1\leq k\leq N$, commute, see~\cite[Lemma~1.1\,e)]{Tul}.
He also shows that if $v$ admits another decomposition $v=\sum_{j=1}^Mv'_j$ satisfying the assumptions of Definition~\ref{df:TulX} then one has
$\prod_{k=1}^Nx(u,\,v_ka)=\prod_{j=1}^Mx(u,\,v'_ja)$, see the discussion following \cite[Lemma~1.1]{Tul}.
Thus, the elements $X_{u,v}(a)$ are well-defined. Obviously, $\phi(X_{u,v}(a))=t(u,\,va)$.

\begin{rk} \label{rk:ID}
It follows from the canonical decomposition~\eqref{eq:canonical} that $vb\in D(u)$ for any $b\in I(u)$ and $v$ such that $v^tu=0$.
In particular, if $I$ is uniquely $a$-divisible and $a^k\in I(u)$, then every $v\in I^n$ satisfying $v^tu=0$ is contained in $D(u)$.
\end{rk}

\begin{lm} \label{xproperties}
For $u$, $w\in B^n$, $v, v' \in D(u)$ such that $u^tw=0$, for $a$, $b\in I(u)$, $c\in B$, $g\in\St(n,\,B)$ the following holds:
\begin{enumerate}
\item $X_{u,vc}(a)=X_{u,v}(ca)$,
\item $X_{uc,v}(ca)=X_{u,vc^2}(a)$,
\item $X_{u,v}(a)X_{u,v'}(a)=X_{u,v+v'}(a)$,
\item $g\,X_{u,wb}(a)g\inv=X_{\phi(g)u,\phi(g)^*wb}(a)$.
\end{enumerate}
\end{lm}

\begin{proof}
The statement of $(\mathrm a)$ is obvious from the definition, $(\mathrm b)$ follows from~\cite[Lemma~1.1\,d)]{Tul}, $(\mathrm c)$ is exactly the statement of~\cite[Lemma~1.3\,a)]{Tul}.

The assertion of $(\mathrm d)$ is proven in~\cite[Lemma~1.3\,b)]{Tul} under the assumption $n\geq5$. 
Tulenbaev remarks that the assertion remains true for $n=4$. 
Indeed, take $z\in B^n$ such that $z^tu=b$ and write the canonical decomposition \eqref{eq:canonical}:
$$(z^tu)w=\sum_{i<j}u_{ij} c_{ij}, \text{ where } u_{ij}=e_iu_j-e_ju_i,\text{ and }c_{ij}=w_iz_j-w_jz_i.$$
Each $u_{ij} c_{ij}$ is orthogonal to $u$ and, since $n\geq4$, it has at least two zero components.
Thus, $X_{u,wb}(a)=\prod_{i<j}x(u,\, u_{ij} ac_{ij})$. 
It suffices to prove the assertion of $(\mathrm d)$ in the special case $g=x_{hk}(r)$.
If $h\neq i,j$ or $\{h,k\}=\{i,j\}$ the vector $\phi(g)^*u_{ij}ac_{ij}$ still has two zero components.

Now, consider the case $j=h$, $i\neq k$.
By~\cite[3.12]{vdK} one has $$g\,x(u,\,u_{ij}ac_{ij})g\inv=x(\phi(g)u,\,\phi(g)^*u_{ij}ac_{ij}).$$
Set $u' = \phi(g)u$, since
$\phi(g)^*u_{ij} = t_{kj}(-r) \cdot u_{ij} = e_iu_j - e_ju_i + e_kru_i =u'_{ij}+u'_{ki}r,$ 
from~\cite[3.11]{vdK} we obtain that
$$x(\phi(g)u,\,\phi(g)^* u_{ij} ac_{ij})=x\left(u',\,u'_{ij}c_{ij}a\right)\cdot x\left(u',\,u'_{ki} rc_{ij}a\right).$$
Since $u_{ij} = - u_{ji}$, $c_{ij}= - c_{ji}$ the case $i=h$, $j\neq k$ formally follows from the previous one.

Decomposing in this fashion each factor $g\,x(u,u_{ij}ac_{ij})g^{-1}$ for which $\phi(g)^*u_{ij}ac_{ij}$ does not have two zero components, 
 we arrive at a product satisfying the requirements of the definition of $X_{\phi(g)u,\phi(g)^*wb}(a)$.
 \end{proof}

We will also need the ``transposed'' analogue of Tulenbaev's elements $X_{u,v}(a)$.
\begin{df} \label{df:TulY}
Let $v\in B^n$, $a\in I(v)$, and $u \in D(v)$, i.\,e. $u = \sum\limits_{i=1}^M u_i$ for some $u_k$ such that $v^tu_k=0$ and each $u_k$ has two zero components.
Set $Y_{u,v}(a)=\prod_{k=1}^Mx(u_ka,\,v)$. 
Similarly to \ref{df:TulX} one can check that this definition does not depend either on the order of factors or on the choice of a decomposition for $u$.
\end{df}

One can repeat van der Kallen's and Tulenbaev's arguments and prove the following transposed version of Lemma~\ref{xproperties}.
We leave the proof of this lemma to the reader.
\begin{lm}
\label{yproperties}
For $u$, $u'$, $w$ and $v\in B^n$, such that $u$ and $u'$ have decomposition as in the above definition, $w^tv=0$, $a$, $b\in I(v)$, $c\in B$, $g\in\St(n,\,B)$ one has
\begin{enumerate}
\item $Y_{uc,v}(a)=Y_{u,v}(ca)$,
\item $Y_{u,vc}(ca)=Y_{uc^2,v}(a)$,
\item $Y_{u,v}(a)X_{u',v}(a)=Y_{u+u',v}(a)$,
\item $g\,Y_{wb,v}(a)g\inv=Y_{\phi(g)wb,\phi(g)^*v}(a)$.
\end{enumerate}
\end{lm}

Finally, it remains to show that for a ``nice'' pair $(u,\,v)$ the elements $X$ and $Y$ coincide.
\begin{lm}\label{xeqy}
Let $u$, $v$, $x$, $y$ be elements of $B^n$ and $b\in I(u)\cap I(v)$, $r\in B$ be such that $u^t v = 0$, $x^ty=b$, $x^tv=0$, $u^ty=0$, $x^tu = 0$, $y^tv=0$.
Then one has $X_{u,vb^4r}(b)=Y_{ub^4r,v}(b).$
\end{lm}
\begin{proof}
Compute $g=[Y_{-xbr,v}(b),\,X_{u,yb}(b)]$ in two different ways using Lemmas~\ref{xproperties}~and~\ref{yproperties}:
$$ g = X_{t(xb^2r,-v)u,\,t(xb^2r,-v)^*yb}(b)X_{u,-yb}(b) = X_{u,\,yb+vb^4r}(b)\, X_{u,-yb}(b) = X_{u,vb^4r}(b),$$
$$ g = Y_{-xbr,v}(b) Y_{t(u,yb^2)xbr,\,t(u,yb^2)^*v}(b)= Y_{-xbr,v}(b)\, Y_{xbr+ub^4r,\,v}(b)=Y_{ub^4r,v}(b).\qedhere$$
\end{proof}

Now, we are all set to construct the desired map $T\colon\St(n,\,B_a,\,I)\rightarrow\St(n,\,B).$

\begin{proof}[Proof of Theorem~\ref{a3map}]
Consider $u=Me_1$, $M\in\E(n,\,B_a)$ and $v\in I^n$ such that $u^tv=0$.
Set $w=M^*e_1$. Since $w^tu=1$ there exist vectors $\tilde w$, $\tilde u\in B^n$ and a natural number $m$ such that 
$$ \lambda_a(\tilde{w})=wa^m,\ \lambda_a(\tilde{u})=ua^m\text{ and moreover }\tilde u^tv=0\text{ and }\tilde w^t\tilde u=a^{2m}.$$
It is clear that $a^{2m}\in I(\tilde u)$, moreover by~\cref{rk:ID} we have $v/a^{3m} \in D(\tilde u)$, therefore we can set $T(F(u,\,v))=X_{\tilde u,v/a^{3m}}(a^{2m})$. 

The first two assertions of Lemma~\ref{xproperties} guarantee that this definition does not depend on the choice of $m$ and the liftings $\tilde u$ and $\tilde w$.
Similarly, we can define $T(S(u,\,v))=Y_{u/a^{3m},\tilde v}(a^{2m})$.
In view of Lemmas~\ref{xproperties}~and~\ref{yproperties} the map $T$ preserves relations \eqref{add4}--\eqref{conj3}.

It remains to check that for $u = Me_1$, $v=M^*e_2$ and $c\in I$ one has
\begin{equation}T(F(u, cv)) = X_{\tilde u,\,\tilde v c/a^{4m}}(a^{2m}) = Y_{\tilde uc/a^{4m},\, \tilde v}(a^{2m}) = T(S(uc, v)). \label{eq:last}\end{equation}
Here $\tilde u, \tilde v \in B^n$ are liftings of $u$ and $v$ such that $\lambda_a(\tilde u)= ua^m$, $\lambda_a(\tilde v)= va^m$ and $\tilde u^t \tilde v = 0$.
We can also assume that $m$ is large enough to ensure that $a^{2m}\in I(\tilde{u})\cap I(\tilde{v})$ and there exist $\tilde{x}$, $\tilde{y}$ such that the following equations hold:
$$\lambda_a(\tilde{x})=Me_3a^m,\ \lambda_a(\tilde{y})=M^*e_3a^m\text{ and }\tilde{x}^t\tilde{y}=a^{2m},\ \tilde{x}^t\tilde{v} = 0,\ \tilde{u}^t\tilde{y} =0,\ \tilde{x}^t\tilde{u} = 0,\ \tilde{y}^t\tilde{v}=0.$$
To obtain~\eqref{eq:last} it remains to apply Lemma~\ref{xeqy} (with $b=a^{2m}$, $r=c/a^{12m}$).
Thus, the map $T$ is well-defined.

Commutativity of the diagram~\eqref{eq:a3diag} follows directly from the definitions of elements $X(u,\,v)$, $Y(u,\,v)$ and $X_{u,v}(a)$, $Y_{u,v}(a)$. \end{proof}

\printbibliography

\end{document}